\documentclass[twoside,
]{amsart}

  \usepackage[utf8]{inputenc}
  \usepackage{amsmath}
  \usepackage{amssymb}
  \usepackage{mathtools}
  \usepackage{mathrsfs}
  \usepackage{bm}
  \usepackage{bbm}
  \usepackage{pifont}
  \usepackage{hyperref}
  
  \usepackage[svgnames,x11names]{xcolor} 

  \hypersetup{
    colorlinks, linkcolor={Chocolate4},
    citecolor={Chocolate4}, urlcolor={DarkOliveGreen4}
  }

  \usepackage{color}
  \usepackage{todonotes}

	\newtheoremstyle{slanted}
	{}
	{}
	{\slshape}
	{}
	{\bfseries}
	{.}
	{ }
	{}
	
	\theoremstyle{slanted}
	\newtheorem{theo}{Theorem}[section]
	\newtheorem{prop}[theo]{Proposition}

	\newtheorem{lemma}[theo]{Lemma}
	\newtheorem{definition}[theo]{Definition}
	\newtheorem{corollary}[theo]{Corollary}
	\newtheorem{remark}[theo]{Remark}
	\newtheorem{example}[theo]{Example}

	\newcommand{\egdef}{:=}
	\DeclareMathOperator{\Id}{Id}	
	\newcommand{\tend}[3][]{\xrightarrow[#2\to#3]{#1}}
		
	\newcommand{\EE}{\mathbb{E}}

	\newcommand{\ind}[1]{\mathbbmss{1}_{#1}} 
	\newcommand{\ZZ}{\mathbb{Z}}

	\newcommand{\RR}{\mathbb{R}}
	
	\newcommand{\PP}{\mathbb{P}}
	
	\newcommand{\A}{\mathcal{A}}
	
	\newcommand{\F}{\mathscr{F}}
	
	\renewcommand{\P}{\mathscr{P}}

\title{Ergodic Poisson splittings}
\author{\'{E}lise Janvresse, Emmanuel Roy and Thierry de la Rue}
\address{\'Elise Janvresse: 
Laboratoire Amiénois de Mathématique Fondamentale et Appliquée, CNRS-UMR 7352, Université de Picardie Jules Verne, 33 rue Saint Leu, F80039 Amiens cedex 1,
France.}
\email{Elise.Janvresse@u-picardie.fr}
\address{Emmanuel Roy: Laboratoire Analyse, Géométrie et Applications, Université Paris 13 Institut Galilée,
99 avenue Jean-Baptiste Clément
F93430 Villetaneuse, France.}
\email{roy@math.univ-paris13.fr}
\address{Thierry de la Rue:
Laboratoire de Mathématiques Rapha\"el Salem,
Université de Rouen, CNRS,
Avenue de l'Université,
F76801 Saint \'Etienne du Rouvray, France.}
\email{Thierry.de-la-Rue@univ-rouen.fr}
\thanks{Research partially supported by French research group GeoSto
(CNRS-GDR3477)}

\begin{document}
\bibliographystyle{amsplain}

\begin{abstract}
  In this paper we study splittings of a Poisson point process which are equivariant under a conservative transformation.
  We show that, if the Cartesian powers of this transformation are all ergodic, the only ergodic splitting is the obvious one,
  that is, a collection of independent Poisson processes.
  We apply this result to the case of a marked Poisson process: under the same hypothesis, the marks are necessarily independent of 
  the point process and i.i.d.
  Under additional assumptions on the transformation, a further application is derived, giving a full description of the structure of a 
  random measure invariant under the 
  action of the transformation.  
\end{abstract}


\maketitle

\section{Introduction}

Thinning and splitting are classical operations when studying point processes. Thinning consists in removing points according to some rule, whereas the related notion of splitting means decomposing the point process as the sum of several other point processes. 
It is well known that thinning a Poisson point process by choosing to remove points according to  independent coin tosses yields a new Poisson process of lower (but proportional) intensity. 
Moreover, this procedure gives rise to a splitting 
of the original Poisson process into a sum of two independent Poisson processes.

In recent years, some new results on thinnings of Poisson process have emerged. 
In particular, it is shown in \cite{Ball} that it is possible to deterministically choose points from a homogeneous Poisson 
point process on $\RR$ to get another homogeneous Poisson process of lower intensity. Moreover, it is possible to proceed 
in a translation equivariant way. This result has been further refined in \cite{HolLyoSoo} by extending it to $\RR^d$ and 
replacing translation equivariance by  isometry equivariance. Moreover, the remaining points were also shown to form a 
homogeneous Poisson point process.

The isometry or translation equivariance plays a key role here as Meyerovitch showed in \cite{Meyerovitch2013} that it is 
not possible to deterministically thin a Poisson point process in an equivariant way with respect to a transformation which 
is conservative ergodic on the base space (the translations of $\RR^d$ yields dissipativity).

In the present paper, we are also interested in thinnings and splittings of Poisson processes 
which are equivariant under some dynamics.
The difference with \cite{Ball,HolLyoSoo} is that we consider 
thinnings/splittings which are equivariant with respect to a conservative transformation. Moreover, 
contrary to the above-mentioned result of  \cite{Meyerovitch2013}, we allow additional 
randomness (yet keeping ergodicity) in the thinning/splitting procedure.
We get the following result, phrased in our terminology where equivariance and homogeneity are expressed through the concepts of \emph{$T$-point processes} and \emph{$T$-splittings}:
If $T$ is an infinite measure preserving map with \emph{infinite ergodic index} (\textit{i.e.} all its Cartesian products are ergodic), 
then any ergodic $T$-splitting of a Poisson $T$-point process
yields independent Poisson $T$-point processes (Theorem~\ref{thm:poissonSplitting}).

In the second part of the paper we derive some applications. In Section~\ref{sec:marked}, we show that, for $T$ with infinite 
ergodic index, the only way to get an ergodic marked $T$-point process out of a
Poisson $T$-point process is to take i.i.d. marks, independent of the underlying process (Theorem~\ref{prop:markedPoisson}).
Then we come back to the problem we raised in~\cite{sushis}. In that paper we gave conditions on $T$ under which an ergodic 
$T$-point process with moments of all orders is necessarily a cluster-Poisson process (see~\cite[p.~175]{DaleyVereJonesI}), best described as an independent superposition of shifted Poisson processes (a so-called \emph{SuShi}).
We extend this result in Section~\ref{sec:random_measures} to general random measures (Proposition~\ref{prop:continuous} and Theorem~\ref{thm:discrete}).
This more general framework allows to simplify and improve some disjointness results from~\cite{sushis} in the last section.

\subsection{$T$-random measures and $T$-point processes}

Let $X$ be a complete separable metric space and $\A$ be its Borel $\sigma$-algebra.  
Define $\widetilde{X}$ to be the space of \emph{boundedly finite measures} (also called  \emph{locally finite measures}) on $\left(X,{\A}\right)$, 
that is to say measures giving finite mass to any bounded Borel subset of $X$.
We refer to~\cite{Kallenberg2017} for the topological properties of $\widetilde{X}$. In particular, $\widetilde{X}$ can be turned into a complete separable metric space,
that we equip with its Borel $\sigma$-algebra $\widetilde{\A}$.

We denote by $X^{*}\subset \widetilde{X}$ the subspace of simple counting measures, \textit{i.e.} whose elements are of the form 
\[
\xi=\sum_{i\in I}\delta_{x_{i}},
\]
where $I$ is at most countable, $x_i\neq x_j$ whenever $i\neq j$, and any bounded subset $A\subset X$ contains finitely many points of the family  $\{x_i\}_{i}$. 
We define $\A^*$ as the restriction of $\widetilde{\A}$ to $X^*$.

\medskip

Throughout the paper, we fix a boundedly finite and continuous measure $\mu$ on $X$ with $\mu(X)=\infty$, and an invertible transformation $T$ on $X$ preserving $\mu$. 
We set
\[{\A}_{f}:=\left\{ A\in{\A},\;0<\mu\left(A\right)<+\infty\right\}.\]

For any measure $\xi$ on $X$, we define $T_*(\xi)$ as the pushforward measure of $\xi$ by $T$: for any $A\in\A$,
\[
  T_*(\xi)(A):=\xi(T^{-1}A).
\]
In particular, if $\xi=\sum_{i\in I}\delta_{x_{i}}$, then $T_{*}\left(\xi\right)=\sum_{i\in I}\delta_{T\left(x_{i}\right)}$.

As we already noticed in~\cite{sushis}, the property of bounded finiteness may be lost by the action of $T$. 
Nevertheless, if $m$ is a $\sigma$-finite measure on $\widetilde{X}$ which is concentrated on $\bigcap_{n\in\ZZ}T_*^{-n}\widetilde{X}$, 
it makes sense to consider the $T_*$-invariance of $m$. In this case, $T_*(\xi)\in \widetilde{X}$ for $m$-almost all $\xi\in \widetilde{X}$, 
and $(\widetilde{X},\widetilde{\A},m,T_*)$ is an invertible measure preserving dynamical system.

\medskip


The following definition generalizes the notion of $T$-point process introduced in~\cite{sushis}.

\begin{definition}
A \emph{$T$-random measure} is a random variable $N$ defined on some probability space $\left(\Omega,\F ,\mathbb{P}\right)$ with
values in  $\left(\widetilde{X},\widetilde{\A}\right)$
such that
\begin{itemize}
\item for any set $A\in{\A}$, $N\left(A\right)=0$ $\mathbb{P}$-a-s.
whenever $\mbox{\ensuremath{\mu}}\left(A\right)=0$;
\item there exists a measure preserving invertible transformation $S$ on
$\left(\Omega,\F ,\mathbb{P}\right)$, such that for any set $A\in{\A}$,
$N\left(A\right)\circ S=N\left(T^{-1}A\right)$.
\end{itemize}
\end{definition}

Thus, a $T$-random measure $N$ implements a factor relationship between the dynamical systems $\left(\Omega,\F ,\mathbb{P},S\right)$
and $\left(\widetilde{X},\widetilde{\A},m,T_{*}\right)$, where $m$
is the pushforward measure of $\,\mathbb{P}$ by $N$.
We say that \emph{$N$ is ergodic} whenever $\left(\widetilde{X},\widetilde{\A},m,T_{*}\right)$ is ergodic. In particular $N$ is ergodic as soon as $\left(\Omega,\F ,\mathbb{P},S\right)$ is itself ergodic.

The \emph{intensity} of  a $T$-random measure $N$ is the $T$-invariant measure on $X$ defined by the formula 
$A\in{\A}\mapsto\mathbb{E}\left[N\left(A\right)\right]$. It  is absolutely continuous with
respect to $\mu$ and if it is $\sigma$-finite, it is a multiple of $\mu$, by ergodicity of $\left(X,{\A},\mu,T\right)$. 
In this  case, we say that $N$ is \emph{integrable}.
More generally, the higher order moment measures can be defined as follows.

\begin{definition}
Let $n\ge1$. A $T$-random measure $N$ on $\left(\Omega,\F ,\mathbb{P},S\right)$
is said to have \emph{moments of order $n$} if, for  all  bounded $A\in{\A}$, $\mathbb{E}\left[\left(N\left(A\right)\right)^{n}\right]<+\infty$.
In this case,  the formula
\[
M_{n}^{N}(A_1\times\cdots\times A_n) \egdef \mathbb{E}\left[N\left(A_{1}\right)\times\cdots\times N\left(A_{n}\right)\right]
\]
($A_1,\ldots,A_n\in\A$) defines a boundedly finite $T^{\times n}$-invariant measure
$M_{n}^{N}$ on $\left(X^{n},{\A}^{\otimes n}\right)$ called
the \emph{$n$-order moment measure}.

A $T$-random measure with moments of order $2$ is said to be \emph{square integrable}.
\end{definition}

A \emph{$T$-point process} is a $T$-random measure taking values in $X^*$. In this case, for $\omega\in\Omega$, 
we identify $N(\omega)$ with the corresponding set of points in $X$.  
The most important $T$-point processes are Poisson point processes, let
us recall their definition.
\begin{definition}\label{def:Poisson}
A random variable $N$ with values in $(X^*, \A^*)$ is a \emph{Poisson point process of intensity $\mu$} if 
for any $k\ge1$, for any mutually disjoint sets $A_{1},\dots,A_{k}\in{\A}_{f}$,
the random variables $N\left(A_{1}\right),\dots,N\left(A_{k}\right)$
are independent and Poisson distributed with respective parameters
$\mu\left(A_{1}\right),\dots,\mu\left(A_{k}\right)$. 
\end{definition}

Such a process always exists, and its distribution $\mu^{*}$ on $X^*$ is uniquely determined by the preceding conditions.
%
%

Since $T$ preserves $\mu$, one easily checks that $T_{*}$ preserves~$\mu^{*}$. And defining $N$ on the probability space $\left(X^{*},{\A}^{*},\mu^{*}\right)$ as the identity map provides an example of a $T$-point process, the underlying measure-preserving transformation being $S=T_*$ in this case.

\begin{definition}
The probability-preserving dynamical system $\left(X^{*},{\A}^{*},\mu^{*},T_{*}\right)$ is called
the \emph{Poisson suspension} over the base $\left(X,{\A},\mu,T\right)$.
\end{definition}

The basic result (see \textit{e.g.}~\cite{Roy2007}) about Poisson suspensions states that
$\left(X^{*},{\A}^{*},\mu^{*},T_{*}\right)$ is ergodic (and
then weakly mixing) if and only if there is no $T$-invariant
set in ${\A}_{f}$. In particular this is the case if $\left(X,{\A},\mu,T\right)$
is ergodic and $\mu$ infinite.

We also recall the classical isometry formula that will be useful several
times in this paper: for $f,g\in L^{1}\left(\mu\right)\cap L^{2}\left(\mu\right)$,
\begin{multline}
  \label{eq:isometry}
  \EE_{\mu^*}\left[ \left( \int_X f(x)\, N(dx) - \int_X f(x)\, \mu(dx) \right)  
  \left( \int_X g(x)\, N(dx) - \int_X g(x)\, \mu(dx) \right)  \right]  \\
  =\int_{X}f\left(x\right)g\left(x\right)\mu\left(dx\right).
\end{multline}

%
%
%
%
 
 \begin{remark} 
The notion of $T$-random measure with intensity $\mu$ can be interpreted in terms
of \emph{quasifactors} as introduced by Glasner and Meyerovitch.
Glasner defined in~\cite{Glasner1983} a quasifactor of a  probability
measure preserving system $\left(X,\mathcal{A},\mu,T\right)$ as a probability measure preserving system 
$\left(\widetilde{X},\widetilde{\mathcal{A}},m,T_{*}\right)$
where $\mathbb{E}_{m}\left[N\left(A\right)\right]=\mu\left(A\right)$. 
(Here, $N$ is the random variable defined by the identity on $\widetilde{X}$, and in the case where $\mu$ is a probability measure, 
$m$ is in fact concentrated on the subset of probability measures on $X$.)
Meyerovitch in~\cite{Meyerovitch2011} extended this definition to the case where $\mu$
is infinite (but $m$ is still a probability measure).
Thus
\begin{itemize}
  \item Poisson suspensions appear as natural example of ergodic quasifactors. 
  \item any $T$-random measure $N$ with intensity $\mu$ on $(\Omega,\F,\PP,S)$ gives rise to the quasifactor defined by $m:=N_*(\PP)$;
  \item any quasifactor $\left(\widetilde{X},\widetilde{\mathcal{A}},m,T_{*}\right)$ is associated to the $T$-random measure $N:=\Id$ on the probability space $\left(\widetilde{X},\widetilde{\mathcal{A}},m,T_{*}\right)$.
\end{itemize}

In Section~\ref{sec:infinite_quasifactors}, we will consider yet another case, namely when $m$
is an infinite measure, and use the terminology \emph{$\infty$-quasifactor}
in this case.
\end{remark}

\subsection{Splittings}
\begin{definition}
Let $N$ be a $T$-point process defined on the dynamical system $\left(\Omega,\mathcal{F},\mathbb{P},S\right)$. 
For $1\le k\le \infty$, a \emph{$T$-splitting of order $k$ of $N$} is a finite or countable family of $T$-point
processes $\left\{ N_{i}\right\} _{0\le i< k}$ defined on $\left(\Omega,\mathcal{F},\mathbb{P},S\right)$
so that $N=\sum_{0\le i <k}N_{i}$. 
\end{definition}

We use the terminology ``$T$-splitting" in the above definition to insist on the equivariance of the splitting with respect to the underlying transformation $T$. 
In this paper however, whenever we deal with $T$-point processes, the splittings are always assumed to be $T$-splittings, and we will omit the ``$T$" in the sequel.

\medskip

The splitting is said to be \emph{ergodic} if the joining generated by $\left\{ N_{i}\right\} _{0\le i< k}$ is ergodic.
In the situation $N^\prime\le N$, the usual terminology considers $N^\prime$ as a \emph{thinning} of $N$, and we get $(N^\prime , N-N^\prime)$ as a splitting of order $2$.

\medskip

A \emph{Poisson splitting of order $k$} is a splitting such that $\left\{ N_{i}\right\} _{0\le i< k}$ are independent Poisson processes.
(In this case, $N$ itself has to be a Poisson $T$-point process.)

\medskip

As we mentioned in the introduction, under the assumption of conservativity and ergodicity of $T$, Meyerovitch~\cite{Meyerovitch2013} proved that, in the canonical space
$(X^*,\A^*,\mu^*,T_*)$ of the Poisson suspension, there exists no splitting of the canonical Poisson $T$-point process. But of course 
a splitting can exist in a larger probability space (for example a product space, in which we can find two independent Poisson $T$-point processes).

\subsection{Properties of Cartesian powers of $T$}

 Let us end the introduction by saying a few words about some properties of the Cartesian products of the underlying infinite measure preserving transformation that we will be dealing with. 
 
 It is well known, in the finite measure setting, that a weak mixing transformation has infinite ergodic index (all its Cartesian powers are ergodic). Moreover, if it is not weak mixing then its Cartesian square is already not ergodic. The situation is therefore pretty clear. In the infinite measure case however, the picture is definitely not as simple. It was first observed in~\cite{KakPar} that all intermediate situations may occur: For any $k\ge 1$, there is a transformation with first $k$ Cartesian products ergodic whereas the $k+1$ Cartesian product is not.
 The same authors also give examples of infinite measure preserving transformations with infinite ergodic index. All these examples fall into the Markov chain category.
 
 Since then, as the zoo of infinite measure preserving transformations developed, various examples of transformations having infinite ergodic index or not were built (see~\cite{AFS1997} where  the so-called \emph{infinite Chacon Transformation} --- an infinite measure preserving version of the classical Chacon transformation --- is shown to have infinite ergodic index).
 
 In the last part of the paper, we will assume a much stronger property of $T$, 
 which can be viewed as a strong version of the Radon minimal self-joinings property introduced by Danilenko in~\cite{Danilenko2018}, 
 and roughly saying that the Cartesian powers of the transformation admit as few invariant measures as possible (see Definition~\ref{def:haddock}). 
 An example of a transformation enjoying this property, the \emph{nearly finite Chacon transformation}, is described in~\cite{nfc}.

\section{Splitting of Poisson $T$-point processes}

For each $n\ge1$, we denote by $\mathscr{P}_{n}$ the set of all
partitions of $\left\{ 1,\dots,n\right\} $. Given $\pi\in\mathscr{P}_{n}$,
we define a measure on $X^{n}$ by
\[
m_{\pi}\left(A_{1}\times\cdots\times A_{n}\right):=\prod_{P\in\pi}\mu\left(\cap_{i\in P}A_{i}\right).
\]
For a given $n$, these measures are $T^{\times n}$-invariant and mutually singular. The measure corresponding to the trivial partition with a 
single atom is called the $n$-diagonal measure, and is concentrated on $D_n:=\{(x_1,\ldots,x_n)\in X^n: x_1=\cdots=x_n\}$.

It is well known that the $n$-order moment measure of the Poisson process of
intensity $\mu$ takes the form
\[
\sum_{\pi\in\mathscr{P}_{n}}c_{\pi}m_{\pi}.
\]
for positive coefficients $c_{\pi}$, $\pi\in\mathscr{P}_{n}$.
Moreover, if $N$ is a Poisson process of intensity $\alpha\mu$, then the $n$-order moment measure equals
\[
\sum_{\pi\in\mathscr{P}_{n}}c_{\pi}\alpha^{\#\pi}m_{\pi}.
\]

In the context of a $T$-point process, it turns out that the existence for each moment measure of a decomposition 
as a linear combination of the measures $m_\pi$  characterizes Poisson processes.
This is the object of the following theorem, whose proof is hidden in Theorem~3.2 in~\cite{sushis}. Although the argument is almost word-for-word
the same, we repeat the proof here since the assumptions are far more general, and it would be cumbersome to explain
the differences without giving all the details.

\begin{theo}
Let $\left(X,\mathcal{A},\mu,T\right)$ be an infinite measure preserving dynamical system with no invariant set of positive finite measure.
Let $N$ be a $T$-point process with moments of all orders defined on $\left(\Omega,\mathcal{F},\mathbb{P},S\right)$. 

Then $N$ is a Poisson process if and only if $N$ is ergodic and for all $n \ge 1$, there exist nonnegative numbers $\alpha_\pi$ such that
\[
M_{n}^{N}=\sum_{\pi\in\mathscr{P}_{n}}\alpha_{\pi}m_{\pi}.
\]
\end{theo}

\begin{proof}

Only one direction needs to be detailed.

We can assume that $\left(\Omega,\mathcal{F},\mathbb{P},S\right)$ is ergodic. For $n=1$, we obtain the intensity of $N$ as a multiple of $\mu$, say $\alpha\mu$ for some
$\alpha\ge0$.

We first point out that for each $n\ge1$, the weight of the $n$-diagonal
measure is~$\alpha$ (this is valid for any point process of intensity
$\alpha\mu$). Indeed, using a set $A\in{\A}_{f}$, and
$\left(\left(A_{i}^{\ell}\right)_{1\le i\le p_\ell}\right)_{\ell\ge1}$ a generating sequence of
partitions of $A$, we get 
\begin{multline*}
  \sum_{i=1}^{p_\ell}M_{n}^N\left(A_i^{\ell}\times\cdots\times A_i^{\ell}\right) = \mathbb{E}\left[\sum_{i=1}^{p_\ell}N\left(A_i^{\ell}\right)\cdots N\left(A_i^{\ell}\right)\right]\\
  \tend{\ell}{\infty}  
  \mathbb{E}\left[N\left(A\right)\right]=\alpha\mu\left(A\right)=\alpha\mu\left(A\cap\cdots\cap A\right).
\end{multline*}
On the other hand,
\begin{multline*}
  \sum_{i=1}^{p_\ell}M_{n}^N\left(A_i^{\ell}\times\cdots\times A_i^{\ell}\right) =M_{n}^N\left(\bigsqcup_{i=1}^{p_\ell}A_i^{\ell}\times\cdots\times A_i^{\ell}\right)\\
  \tend{\ell}{\infty} M_{n}^N\left(D_{n}\cap A\times\cdots\times A\right).
\end{multline*}
Therefore $M_{n}^N\left(D_{n}\cap A\times\cdots\times A\right)=\alpha\mu\left(A\cap\cdots\cap A\right)$.
Since the $n$-diagonal measure is the only measure $m_\pi$ charging $D_n$, this  implies as claimed  that the weight of the $n$-diagonal measure is $\alpha$.

We now want to prove by induction that, for all $n\ge1$, $M_{n}^N$
is the $n$-order moment measure of a Poisson process of intensity $\alpha \mu$.
The property is of course satisfied for $n=1$. Let us assume it is satisfied
up to some $n\ge1$, and let $A_{1},\dots,A_{n+1}$ be sets in ${\A}_{f}$.
Pick a nonempty subset $K\subsetneq\left\{ 1,\dots,n+1\right\}$.
By the ergodic theorem, we get
\begin{multline}
\label{eq:1}
 \frac{1}{\ell}\sum_{k=1}^{\ell}\mathbb{E}\left[\prod_{i\in K}N\left(A_{i}\right)\left(\prod_{i\in K^{c}}N\left(A_{i}\right)\circ S^{k}\right)\right]\\
 \tend{\ell}{\infty}  \mathbb{E}\left[\prod_{i\in K}N\left(A_{i}\right)\right]\mathbb{E}\left[\prod_{i\in K^{c}}N\left(A_{i}\right)\right]\\
  =  M^N_{\# K }\left(\bigtimes_{i\in K}A_{i}\right)M^N_{(n+1-\# K)}\left(\bigtimes_{i\in K^{c}}A_{i}\right).
\end{multline}
On the other hand,
\begin{align}
  \label{eq:2}& \frac{1}{\ell}\sum_{k=1}^{\ell}\mathbb{E}\left[\prod_{i\in K}N\left(A_{i}\right)\left(\prod_{i\in K^{c}}N\left(A_{i}\right)\circ S^{k}\right)\right]\\
  \nonumber = & \frac{1}{\ell}\sum_{k=1}^{\ell}M^N_{n+1}\left(T^{-\epsilon_{k}\left(1\right)}A_{1}\times\cdots\times T^{-\epsilon_{k}\left(n+1\right)}A_{n+1}\right)\\
  \nonumber = & \sum_{\pi\in\P_{n+1}}\alpha_{\pi}\frac{1}{\ell}\sum_{k=1}^{\ell}m_{\pi}\left(T^{-\epsilon_{k}\left(1\right)}A_{1}\times\cdots\times T^{-\epsilon_{k}\left(n+1\right)}A_{n+1}\right)
\end{align}
where $\epsilon_{k}\left(i\right)\egdef 0$ if $i\in K$, and $\epsilon_{k}\left(i\right)\egdef k$ otherwise. Coming back to the definition of $m_\pi$, we write
\[
  m_\pi\left(T^{-\epsilon_{k}\left(1\right)}A_{1}\times\cdots\times T^{-\epsilon_{k}\left(n+1\right)}A_{n+1}\right)
  = \prod_{P\in\pi} \mu\left( \bigcap_{i\in P} T^{-\epsilon_k(i)} A_i\right).
\]
Observe that, if $K$ is a union of atoms of $\pi$, we have for any $1\le k\le \ell$
\begin{equation}
 \label{eq:4} m_\pi \left(T^{-\epsilon_k\left(1\right)}A_{1}\times\cdots\times T^{-\epsilon_k\left(n+1\right)}A_{n+1}\right)=m_{\pi}\left(A_{1}\times\cdots\times A_{n+1}\right).
\end{equation}

Otherwise, there exists an atom $P\in\pi$ containing indices $i\in K$ and $j\notin K$, hence with $\epsilon_{k}\left(i\right)=0$ and $\epsilon_{k}(j)=k$. We get that for some constant $C$,
\[
  m_\pi\left(T^{-\epsilon_{k}\left(1\right)}A_{1}\times\cdots\times T^{-\epsilon_{k}\left(n+1\right)}A_{n+1}\right) \le C \mu(T^{-k}A_j\cap A_i).
\]
But, since there is no $T$-invariant set of positive finite measure, 
\[
  \frac{1}{\ell}\sum_{k=1}^{\ell}\mu\left(T^{-k}A_j\cap A_i\right) \tend{\ell}{\infty}0.
\]
Let $\P_{n+1}^{K}$ be the set of partitions $\pi\in\P_{n+1}$ where $K$ is a union of atoms of $\pi$.  The above proves that, in the limit as $\ell\to\infty$, the contribution {in~\eqref{eq:2}} of all partitions $\pi\in\P_{n+1}\setminus\P_{n+1}^{K}$ vanishes. Thus we get, using~\eqref{eq:1}, \eqref{eq:2} and \eqref{eq:4},
\[
M^N_{\# K }\left(\bigtimes_{i\in K}A_{i}\right)M^N_{(n+1-\# K)}\left(\bigtimes_{i\in K^{c}}A_{i}\right)=\sum_{\pi\in\P_{n+1}^K} \alpha_{\pi}m_{\pi}\left(A_{1}\times\cdots\times A_{n+1}\right).
\]
Since $\emptyset\neq K\subsetneq\left\{ 1,\dots,n+1\right\}$, the decompositions of $M^N_{\# K}$ and $M^N_{(n+1-\# K)}$ only involve the coefficients $\alpha_\pi$, $\pi\in\P_1\cup\cdots\cup\P_n$. Using the mutual singularity of the measures on both sides of the above equality, we see  that all
the coefficients $\alpha_{\pi}$, $\pi\in\P_{n+1}^{K}$, are
completely determined by the coefficients corresponding to partitions in $\P_1\cup\cdots\cup\P_n$. Moreover, the above argument is valid
in particular when $N$ is the Poisson process  of intensity $\alpha\mu$. By letting
$K$ run through all strict subsets of $\left\{ 1,\dots,n+1\right\}$, and using the induction hypothesis,
we identify all but one coefficients of the decomposition
of $M^N_{n+1}$ as those of the Poisson point process of intensity
$\alpha \mu$. The only coefficient that cannot be determined by this method
is the one associated to the trivial partition of $\{1,\ldots,n+1\}$ into a single atom. But this corresponds  to the $(n+1)$-diagonal measure, and we already
know that this coefficient is $\alpha$.

We have proved that the moment measures of any order of $N$ are those of a Poisson
point process of intensity $\alpha \mu$. 
Lemma~3.1 in~\cite{sushis} ensures then that $N$ is  a Poisson
point process of intensity $\alpha \mu$.
\end{proof}

Observe that the conclusion of the proof fails if one does not assume the ergodicity of the point process $N$: think of a mixture of two Poisson point processes
with different intensities.

\smallskip

The action of $T^{\times n}$ on $\left(X^n,\mathcal{A}^{\otimes n},m_{\pi}\right)$ is isomorphic 
to $\left(X^{\#\pi},\mathcal{A}^{\otimes\#\pi},\mu^{\otimes\#\pi},T^{\times\#\pi}\right)$. It is 
therefore ergodic if we assume that $T$ has infinite ergodic index.
With this assumption, we get the following easy consequence for a thinning of Poisson $T$-point process.

\begin{prop}
\label{prop:poissonThinning}
Assume $T$ has infinite ergodic index. 
Let $N$ and $N^{\prime}$ be $T$-point processes defined on the system $\left(\Omega,\mathcal{F},\mathbb{P},S\right)$ such that $N$ is Poisson of intensity $\mu$ and $N^{\prime}\le N$. 
If  $N^{\prime}$ is ergodic, then it is Poisson of intensity $\alpha\mu$ for some $0\le\alpha\le1$.
\end{prop}
\begin{proof}
If $T$ has infinite ergodic index, all the measures $m_{\pi}$, $\pi\in\mathscr{P}_{n}$,
$n\ge1$, are ergodic with respect to $T^{\times n}$, therefore, the formula $\sum_{\pi\in\mathscr{P}_{n}}c_{\pi}m_{\pi}$
is precisely the ergodic decomposition of the moment measure $M_{n}^{N}$ of the Poisson process $N$, with respect to $T^{\times n}$.

Now, as $N^{\prime}\le N$, we get for each $n\ge1$, $N^{\prime}\otimes\cdots\otimes N^{\prime}\le N\otimes\cdots\otimes N$, and thus $M_{n}^{N^{\prime}}=\mathbb{E}\left[N^{\prime}\otimes\cdots\otimes N^{\prime}\left(\cdot\right)\right]\le\mathbb{E}\left[N\otimes\cdots\otimes N\left(\cdot\right)\right]=M_{n}^{N}$.
We therefore obtain the ergodic decomposition of $M_{n}^{N^{\prime}}$
in the form
\[
M_{n}^{N^{\prime}}=\sum_{\pi\in\mathscr{P}_{n}}\alpha_{\pi}m_{\pi}
\]
for some non negative numbers $\alpha_{\pi}$.
We can now apply the preceding theorem to the ergodic $T$-point process $N^{\prime}$ to get the result.
\end{proof}

\begin{example}
Let us describe an example where the conclusion of the Proposition
fails. Consider the so-called ``homogeneous Poisson process'', that
is, the classical Poisson process on the real line with intensity equal to the Lebesgue measure. 
The base transformation is the translation $T:\:x\mapsto x+1$ (in particular, the Poisson $T$-point process is ergodic). Now we can
form a thinning $N^{\prime}$ by keeping the points of $N$ that are
separated by at least $\kappa>0$ from every other points. It is easy
to see that $N^{\prime}$ is an ergodic $T$-point process but obviously
not a Poisson process. Here the proposition does not apply, not really because $T$ is not ergodic (in fact $T$ can be embedded in an ergodic action of $\RR$), but because the Cartesian powers fail to be ergodic (even if we consider the $\RR$-action).
\end{example}

For the proof of the next theorem, we shall need some definitions and the following lemma which is  a particular case of Lemma~2.6 in \cite{sushis}.

\begin{definition}
 We say that a $T$-point process $N$  defined on $\left(\Omega,\F ,\mathbb{P},S\right)$ is \emph{free} 
if for $\PP$-almost all $\omega$, for all $k\in\mathbb{Z}^{*}$,
$N(\omega) \cap N\left(S^{k}\omega\right)= \emptyset$.

Two $T$-point processes $N_{1}$ and $N_{2}$ defined on $\left(\Omega,\F ,\mathbb{P},S\right)$
are said to be \emph{dissociated} if, for $\PP$-almost all $\omega$, for all $k\in\mathbb{Z}$,
$N_{1}(\omega) \cap N_{2}\left(S^{k}\omega\right)= \emptyset$.
\end{definition}

Note that if $N$ is a Poisson $T$-point process, then $N$ is free (Proposition~2.7 in~\cite{sushis}).

%

\begin{lemma}
\label{lemma:graph measures}Let $N_{1},\dots,N_{n}$ be $n$ $T$-point
processes defined on the ergodic system $\left(\Omega,\F ,\mathbb{P},S\right)$, having moments of all orders.
Assume there exist a real number $c>0$ and some $\pi\in\P_n$ such that for any sets $A_{1},\dots,A_{n}$ in ${\A}_{f}$,
\[
\mathbb{E}\left[N_{1}\left(A_{1}\right)\cdots N_{n}\left(A_{n}\right)\right] \ge
c m_\pi(A_1\times\cdots\times A_n).
\]
Then, for any atom $P\in\pi$, any $A\in{\A}_{f}$, 
\[
  \PP\left(A\cap\bigcap_{i\in P} N_i\neq \emptyset\right) > 0.
\]
In particular, for $i,j\in P$, the processes $N_i$ and $N_j$ are not dissociated.
\end{lemma}

\begin{theo}
\label{thm:poissonSplitting}
If $T$ has infinite ergodic index, then any ergodic splitting of a Poisson $T$-point process is Poisson.
\end{theo}

\begin{proof}

We start with a splitting of order $k< \infty$ of a Poisson $T$-point process $N$ defined on the ergodic dynamical system $\left(\Omega,\F ,\mathbb{P},S\right)$, that is, we have $k$ $T$-point process $N_1,\dots,N_k$  so that $N=N_1+\dots+N_k$.

From Proposition~\ref{prop:poissonThinning}, the $N_j$'s are Poisson processes with respective intensities $\alpha_1\mu,\dots,\alpha_k\mu$ such that $\alpha_1+\dots+\alpha_k=1$

Let $n_{1},\dots,n_{k}$ be positive numbers, $n:=n_{1}+\dots+n_{k}$,  and let $\{Q_1,\ldots,Q_k\}$ be the partition of $\{1,\ldots,n\}$ in subsets of consecutive integers of respective size $n_1,\ldots,n_k$. For any $\left\{ A_{i}\right\} _{1\le i\le n}$ in ${\A}_{f}$, set
\begin{equation}
  \label{eq:sigma}\sigma(A_1\times\cdots\times A_n)\egdef\mathbb{E}\left[\prod_{j=1}^k\prod_{i\in Q_j} N_j(A_i)\right].
\end{equation}
This defines a $T^{\times n}$-invariant measure $\sigma$ on $\left(X^{n},{\A}^{\otimes n}\right)$, for which, as above, since for all $1\le j \le k$ we have $N_j \le N$, $\sigma \le M_{n}^{N}$. Hence $\sigma$ has at most countably many ergodic components, of the form 
$m_\pi$ for some $\pi\in\P_n$.
Observe that the processes $N_{1},\dots,N_{k}$ are mutually dissociated as, for all $1\le j \le k$, $N_j \le N$  and $N$ is free. Therefore, by Lemma~\ref{lemma:graph measures}, 
if the measure $m_\pi$ appears in the ergodic decomposition of $\sigma$, then $\pi$ refines the partition $\{Q_1,\ldots,Q_k\}$. 
Hence, any ergodic component $m_\pi$ of $\sigma$ has the form
\[
 m_\pi(A_1\times\cdots\times A_n)= \prod_{j=1}^k \prod_{P\in\pi,P\subset Q_j} \mu\left(\bigcap_{i\in P}A_i\right) = \prod_{j=1}^k \nu_j\left(\bigtimes_{i\in Q_j}A_i\right),
\]
where each $\nu_j$ is a $T^{\times n_j}$-invariant measure. In particular, any ergodic component of $\sigma$ is invariant by the transformation
$(x_1,\ldots,x_n)\mapsto(y_1,\ldots,y_n)$, where $y_i\egdef Tx_i$ if $i\in Q_k$, and $y_i\egdef x_i$ otherwise. It follows that $\sigma$ itself is invariant by this transformation, hence the expression 
defining $\sigma(A_1\times\cdots\times A_n)$ on the right-hand side of~\eqref{eq:sigma}
is unchanged if we replace $N_k(A_i)$ by $N_k(T^{-1}A_i)$ for all $i\in Q_k$ simultaneously.
Therefore, we can write for any $\left\{ A_{i}\right\} _{1\le i\le n}$ in ${\A}_{f}$ and any $L\ge1$
\begin{align*}
  \mathbb{E}\left[\prod_{j=1}^k\prod_{i\in Q_j} N_j(A_i)\right] & = \frac{1}{L}\sum_{1\le \ell\le L}
  \mathbb{E}\left[\left(\prod_{j=1}^{k-1}\prod_{i\in Q_j} N_j(A_i)\right) \prod_{i\in Q_k} N_k(T^{-\ell} A_i)\right] \\
  & = \mathbb{E}\left[\left(\prod_{j=1}^{k-1}\prod_{i\in Q_j} N_j(A_i)\right) 
  \left(\frac{1}{L}\sum_{1\le \ell\le L} \prod_{i\in Q_k} N_k\circ S^\ell (A_i)\right) \right] .
\end{align*}
By the ergodic theorem, this converges as $L\to\infty$ to 
\[
  \mathbb{E}\left[\prod_{j=1}^{k-1}\prod_{i\in Q_j} N_j(A_i)\right]
  \mathbb{E}\left[ \prod_{i\in Q_k} N_k (A_i)\right].
\]
A straightforward induction on $k$ then yields the equality
\[
  \mathbb{E}\left[\prod_{j=1}^k\prod_{i\in Q_j} N_j(A_i)\right]
  = \prod_{j=1}^k \mathbb{E}\left[\prod_{i\in Q_j} N_j(A_i)\right],
\]
and this is sufficient to obtain the independence between the Poisson
processes.

The case $k=\infty$ is easily deduced from the finite order case.
\end{proof}

\begin{remark}
  For the conclusion of Theorem~\ref{thm:poissonSplitting} to hold, it is in fact enough to assume only that 
  (with the notations of the proof) all but one of the $T$-point processes $N_j$ are ergodic. Indeed, we can then 
  apply the proof in any ergodic component of the joining $(N_1,\ldots,N_k)$, and see that in each of these
  ergodic components we have the same structure of independent Poisson processes. 
  A posteriori we see that there is only one ergodic component.
  In particular, with the assumptions of Proposition~\ref{prop:poissonThinning}, we get that $N'$ and $(N-N')$ are
  independent Poisson processes and thus form a Poisson splitting.
  Moreover, if we remove the assumption of ergodicity of $N'$ in Proposition~\ref{prop:poissonThinning}, 
  we get that the ergodic components of $(N',N-N')$ are necessarily Poisson splittings.
\end{remark}

\section{Application to marked Poisson point processes}

\label{sec:marked}

Here, we deal with so-called marked point processes. Roughly speaking
a marked point process on $X$ with marks in some measurable space $(K,\mathcal{K})$ is a point process
on $X$ whose points carry some information, a mark, taking values
in  $K$. For a $T$-point process, we require the mark to have the \emph{shadowing property}, meaning that
it follows the point when the dynamics on the point process is applied. We thus consider a \emph{marked $T$-point
process} as a $(T\times \Id)$-point process on the bigger space 
$\left(X\times K,\mathcal{A}\otimes\mathcal{K}\right)$
with intensity measure $\widetilde{\mu}$ that projects on $\mu$.

\begin{theo}
\label{prop:markedPoisson}
Let $\left(X,\mathcal{A},\mu,T\right)$ be an infinite measure preserving
dynamical system with infinite ergodic index. Let 
$N$ be an ergodic $(T\times\Id)$-point process 
on $X\times K$.
Assume that $N_{0}:=N\left(\cdot\times K\right)$
is a Poisson point process with intensity $\mu$. Then $N$ is a Poisson point process with
intensity $\mu\otimes\rho$ where $\rho$ is some probability measure
on $K$.
\end{theo}
\begin{remark}
It is well known that when a Poisson point process is independently
endowed with i.i.d. marks, then the resulting point process on the product
space is a Poisson process whose intensity is the product measure
of the original intensity and the distribution of the marks. 
The above
result means that, for $T$ with infinite ergodic index, the only way to get an ergodic marked $T$-point process out of a
Poisson $T$-point process is precisely to take i.i.d. marks, independent of the underlying
process.
\end{remark}
\begin{proof}
Let us denote by $\widetilde{\mu}$ the intensity of $N$ on $X\times K$.
Let $E=\bigsqcup_{i\in I} A_i\times B_i\subset X\times K$ be a finite union of pairwise disjoint product sets of finite
$\widetilde{\mu}$-measure. Let $(\tilde B_j)_{j\in J}$ be the finite partition of $K$ into nonempty subsets 
generated by the sets $B_i$. For each $j\in J$, denote by $I_j$ the subset of $i\in I$ such that $\tilde B_j \subset B_i$, then
observe that the sets $A_i$, $i\in I_j$ are pairwise disjoint. 
Therefore, $E$ can be also written as the disjoint union
\[
  E = \bigsqcup_{j\in J} \bigsqcup_{i\in I_j}  A_i\times \tilde B_j,
\]
and this  refines the original partition of $E$ into product sets $A_i\times B_i$.
Since the ergodic $T$-point processes $N_{\tilde B_{j}}:=N\left(\cdot\times \tilde B_{j}\right)$, $j\in J$,
form an ergodic splitting of $N_{0}$, we get by Theorem~\ref{thm:poissonSplitting} that they are
independent Poisson $T$-point processes. Since for each $j\in J$ the sets $A_i$, $i\in I_j$, are pairwise
disjoint, the random variables $N(A_i\times \tilde B_j)$, $j\in J$, $i\in I_j$ are independent Poisson  random variables.
Finally, the random variables $N(A_i\times  B_i)$, $i\in I$ are also independent Poisson random variables. 
This is enough to prove  that $N$ is a Poisson point process.

To get the intensity, observe that, for each $B\in\mathcal{K}$, $N\left(\cdot\times B\right)$
is a thinning of $N_{0}$ whose intensity is $\mu$. By ergodicity
of $\left(X,\mathcal{A},\mu,T\right)$, there exists $0\le\rho\left(B\right)\le1$
such that $\rho\left(B\right)\mu$ is the intensity of $N\left(\cdot\times B\right)$.
It is now clear that the map $B\mapsto\rho\left(B\right)$ defines
a probability measure on $\left(K,\mathcal{K}\right)$. Finally $\widetilde{\mu}=\mu\otimes\rho$.
%
%
%
%
%
\end{proof}

\section{Extended SuShis: From simple point processes to general random measures}

\label{sec:random_measures}
In this section, we aim to extend the ``rigidity" result obtained in~\cite{sushis} for
simple point processes to general random measures. The hypothesis
on $\left(X,\mathcal{A},\mu,T\right)$ will be the same as in~\cite{sushis} but
rephrased with the notation introduced earlier.

\begin{definition}
\label{def:haddock}
We say that the infinite measure preserving system $\left(X,\mathcal{A},\mu,T\right)$ has the \emph{(P) property}  if, 
for each $n\ge1$ the following is true:
whenever $\sigma$ is a boundedly finite, $T^{\times n}$-invariant measure on $X^{n}$, 
with marginals absolutely continuous with respect to $\mu$, then $\sigma$ is conservative, and its ergodic
components are all of the form $\left(T^{k_{1}}\times\cdots\times T^{k_{n}}\right)_{*}m_{\pi}$
for some $\pi\in\mathcal{P}_{n}$ and integers $k_{1},\dots,k_{n}$.
  \end{definition}
  
Note that the (P) property implies in particular that $T$ has infinite ergodic index 
(otherwise the ergodic components of some product measure $\mu^{\otimes n}$ would not satisfy the required assumption).

\medskip

Let $\widetilde{X}_{c}\subset\widetilde{X}$ and $\widetilde{X}_{d}\subset\widetilde{X}$
be respectively the spaces of continuous and discrete boundedly finite
measures on $\left(X,\mathcal{A}\right)$.

If $N$ is an integrable $T$-random measure, we can write $N$ as $N_{c}+N_{d}$ where $N_c\in\widetilde{X}_{c}$
and $N_d\in\widetilde{X}_{d}$ (both  $N_{c}$ and  $N_{d}$ can be obtained deterministically from $N$). Of course
$N_{c}$ and $N_{d}$ are still integrable $T$-random measures whose respective
intensities are multiples of $\mu$ (by ergodicity of $T$).
Thanks to this decomposition, we can  study separately the case of a.s. continuous $T$-random measures, 
and the case of a.s. discrete $T$-random measures, which will be the objects of the two following sections.

\subsection{The continuous case}
\begin{prop}
\label{prop:continuous}
Assume that $T$ has the (P) property. 

If $N$ is a square integrable $T$-random measure defined on some ergodic system $\left(\Omega,\F ,\mathbb{P},S\right)$,
whose realizations are a.s. continuous, then there exists $\alpha\ge0$ such that $N$ is constant and a.s. equal to $\alpha\mu$. 

If the realizations are a.s. absolutely continuous with respect to $\mu$, no hypothesis on moments are required to get the same conclusion, 
although we might have $\alpha=+\infty$. 
\end{prop}

\begin{proof}
First assume $N$ is square integrable. We know that its intensity is of the form $\alpha\mu$ for some $\alpha\ge0$, and 
the moment measure of order two, defined on $\A\otimes\A$ by
\[
M_{2}^N\left(A\times B\right):=\mathbb{E}\left[N\left(A\right)N\left(B\right)\right]=\mathbb{E}\left[N\otimes N\left(A\times B\right)\right]
\]
is boundedly finite. This measure is $T\times T$-invariant, and its marginals are absolutely continuous with respect to $\mu$. Moreover, 
as $N$ is a.s. continuous, $N\otimes N$ gives zero measure to the graphs of $T^{k}$, for all $k\in\mathbb{Z}$.
Thanks to property~(P),  $M_{2}^N=\beta\mu\otimes\mu$ for some $\beta\ge0$.
Applying the ergodic theorem, we get for all $A\in\A_f$
\[
\frac{1}{n}\sum_{k=1}^{n}\mathbb{E}\left[N\left(A\right)N\left(A\right)\circ S^{k}\right]\tend{n}{\infty}\mathbb{E}\left[N\left(A\right)\right]^{2}=\alpha^{2}\mu\left(A\right)^{2}.
\]
On the other hand,
\[
\frac{1}{n}\sum_{k=1}^{n}\mathbb{E}\left[N\left(A\right)N\left(A\right)\circ S^{k}\right]
=\frac{1}{n}\sum_{k=1}^{n}M_2^N(A\times T^{-k}A)
=\frac{1}{n}\sum_{k=1}^{n}\beta\mu\left(A\right)\mu\left(T^{-k}A\right),
\]
which by invariance of $\mu$ is equal to $\beta\mu\left(A\right)^{2}$.
Therefore $\beta=\alpha^{2}$ and 
\[
  \mathbb{E}\left[\Bigl(N\left(A\right)-\alpha\mu\left(A\right)\Bigr)^{2}\right]=M_{2}^N\left(A\times A\right)-\alpha^{2}\mu\left(A\right)^{2}=0,
\]
which implies the result.

If we assume that the realizations are a.s. absolutely continuous, then we can write
for all set $A\in\mathcal{A}$
\[
N\left(\omega\right)\left(A\right)=\int_{A}f\left(\omega,x\right)\mu\left(dx\right).
\]
We can therefore define $N_{n}$ by
\[
N_{n}\left(\omega\right)\left(A\right)=\int_{A}\left(f\wedge n\right)\left(\omega,x\right)\mu\left(dx\right).
\]
$N_n$ is still a continuous ergodic $T$-random measure but is now square integrable.
From the first part of the proof, $N_{n}=\alpha_{n}\mu$ a.s. Since
$N_{n}$ increases to $N$, we get $N=\alpha\mu$ a.s. where $\alpha$
is the increasing limit of $\alpha_{n}$.
\end{proof}

\subsection{The discrete case}
%
We  consider the set of sequences
\[
  \ell_1^+(\ZZ)\egdef\left\{  (a_k)_{k\in\ZZ}: \forall k\in\ZZ, a_k\ge 0 \text{ and } \sum_{k\in\ZZ} a_k <\infty  \right\}.
\]
The $\ell_1$ norm turns $\ell_1^+(\ZZ)$ into a complete separable metric space. 
The goal of this section is to obtain the following result:
\begin{theo}
\label{thm:discrete}
Assume $T$ has the (P) property. 
Let $N$ be a nonzero $T$-random measure with moments of all orders defined
on some ergodic system $\left(\Omega,\mathcal{F},\PP,S\right)$ and
whose realizations are almost-surely discrete. Then there exists a
probability distribution $\kappa$ on $\ell_1^+(\ZZ)$
and a positive number $c$ such that $N$ is distributed as
\[
A\mapsto\int_{X\times\ell_1^+(\ZZ)}\sum_{k\in\mathbb{Z}}a_{k}\ind{A}\left(T^{k}x\right)\mathcal{N}\left(dx,d\left\{ a_{k}\right\} _{k\in\mathbb{Z}}\right)
\]
where $\mathcal{N}$ is a Poisson point process on $X\times\ell_1^+(\ZZ)$
with intensity $c\mu\otimes \kappa$.
\end{theo}

This result says that $N$ has a cluster form which can be obtained in the following way: start from a Poisson point process
of intensity $c\mu$, then replace independently each point $x$ output by this Poisson point process
with a random measure (the cluster) of the form
\[
\sum_{k\in\mathbb{Z}}a_{k}\delta_{T^{k}x}
\]
where $\left\{ a_{k}\right\} _{k\in\mathbb{Z}}$ is chosen according
to $\kappa$.

%

\medskip

In the following, replacing if necessary $\mu$ by the intensity of $N$ which is of the form $\alpha\mu$ for some $\alpha>0$ by ergodicity of $T$, we assume that the intensity of $N$ is $\mu$.

\subsubsection{Removing points with small weights}
\label{subsubsection:removing points}
Consider $N$ as in the statement of Theorem~\ref{thm:discrete}. 
For $\epsilon>0$, we define $N_{\epsilon}$ from $N$ by removing  points
with weights less than $\epsilon$. We also define $N_{\epsilon,1}$ as the simple
point process obtained from $N_{\epsilon}$ with all weights set to
$1$. We have for all $A\in\mathcal{A}_{f}$
\[
N_{\epsilon,1}\left(A\right)\le\frac{1}{\epsilon}N_{\epsilon}\left(A\right)\le\frac{1}{\epsilon}N\left(A\right),
\]
therefore $N_{\epsilon}$ and $N_{\epsilon,1}$ are both $T$-random
measures with moments of all orders.

In particular, thanks to Proposition~2.1 in~\cite{sushis}, 
$N_{\epsilon,1}$ almost surely belongs to the subset $\widetilde{X}_{d,f}\subset\widetilde{X}_{d}$
of measures $\nu$
satisfying the following property: 
\[ \forall x\in X,\ \#\left\{ n\in\mathbb{Z}:\ \nu(T^{n}x)>0 \right\} <\infty.
  \]
Of course, $N_\epsilon\in\widetilde{X}_{d,f}$ almost surely as well.

We  construct an injective map $\Phi$ from $\widetilde{X}_{d,f}$
to $\left(X\times\ell_1^+(\ZZ)\right)^{*}$. For $\nu\in\widetilde{X}_{d,f}$, $\Phi(\nu)$ is the simple counting measures supported 
on the following collection of points in $X\times\ell_1^+(\ZZ)$: we select in each orbit seen by $\nu$ the first element $x$ in the orbit with maximal weight, 
and consider the point $\left(x,\left(\nu(T^nx)\right)_{n\in\ZZ}\right)$. In other words,
a point $\left(x,\left( \beta_{n}\right) _{n\in\ZZ}\right)$ belongs
to $\Phi\left(\nu\right)$ if and only if
\begin{itemize}
\item $\beta_{0}>\beta_n$ for each $n<0$; 
\item $\beta_{0}\ge \beta_n$ for each $n\ge 0$; 
\item $\nu\left(\{T^nx\}\right)=\beta_n$ for each $n\in\ZZ$.
\end{itemize}
Observe that we can recover $\nu$ from $\Phi\left(\nu\right)$ by the formula
\begin{equation}
  \label{eq:Phi}
  \nu\left(A\right) =\int_{X\times\ell_1^+(\ZZ)}\sum_{k\in\mathbb{N}}\beta_{k}\ind{A}\left(T^{k}x\right)\Phi\left(\nu\right) \left(dx,d\left\{ \beta_{k}\right\} _{k\in\mathbb{N}}\right).
\end{equation}

Now, $\Phi\left(N_{\epsilon}\right)$ is $\mathbb{P}$-a.s. well defined 
as a $(T\times\Id)$-point process on $X\times\ell_1^+(\ZZ)$. 
Its projection on the first coordinate $\left(\Phi\left(N_{\epsilon}\right)\right)_{0}:=\Phi\left(N_{\epsilon}\right)\left(\cdot\times\ell_1^+(\ZZ)\right)$
is a $T$-point process on $X$ with moments of all orders (since $\Phi\left(N_{\epsilon}\right)\left(\cdot\times\ell_1^+(\ZZ)\right)\le N_{\epsilon,1}\left(\cdot\right)$).
By construction, this $T$-point process is free since we have selected one point in each orbit seen by $N_\epsilon$.

Applying Theorem~3.2 in~\cite{sushis} (and this is where we need the assumptions of moments of all orders), we get that $\left(\Phi\left(N_{\epsilon}\right)\right)_{0}$
is a Poisson process. Then by Theorem~\ref{prop:markedPoisson}, we obtain that $\Phi\left(N_{\epsilon}\right)$
is a Poisson process on $\left(X\times\ell_1^+(\ZZ)\right)$. 
Applying~\eqref{eq:Phi},  we almost surely have
\[
N_{\epsilon}\left(A\right)=\int_{X\times\mathbb{R}_{+}^{\mathbb{Z}}}\sum_{k\in\mathbb{N}}\alpha_{k}\ind{A}\left(T^{k}x\right)\Phi\left(N_{\epsilon}\right)\left(dx,d\left\{ \alpha_{k}\right\} _{k\in\mathbb{N}}\right),
\]
and thus we get the theorem for $N_\epsilon$, for any $\epsilon>0$. 

\medskip

To get the result for $N$, we need to take advantage of the infinite divisibility character of  $N_\epsilon$,
that $N$ will inherit in the limit. We therefore have to recall some general results about infinite divisibility.

\subsubsection{Infinitely divisible random measures}

The notion of infinite divisibility can be defined on any measurable semi-group although we only give it
in the context we are interested in.
\begin{definition}
let $\left(Z,\mathcal{Z}\right)$ be complete separable metric space.
A probability distribution $\sigma$ on $\left(\widetilde{Z},\widetilde{\mathcal{Z}}\right)$
is \emph{infinitely divisible} (ID) if , for any $k\in\mathbb{N}$,
there exists a probability distribution $\sigma_{k}$ such that
\[
\sigma=\underbrace{\sigma_{k}\star\cdots\star\sigma_{k}}_{k\text{ times}},
\]
where $\star$ is the convolution of measures induced by the addition
on $\widetilde{Z}$. By extension, we say that a random measure on
$Z$ is ID if its distribution is.
\end{definition}

The first examples of ID random measures are Poisson measures, their
ID character comes directly from the well known fact that the sum
of two independent Poisson processes on the same space with intensities
$\mu_{1}$ and $\mu_{2}$ is again a Poisson process, but with intensity
$\mu_{1}+\mu_{2}$.

We recall the fundamental representation result that can be found
in~\cite[Theorem~3.20]{Kallenberg2017}.

\begin{theo}
 \label{theo:Kallenberg}
 A probability measure $m$ on $\widetilde{X}$ is ID if and only if there exist
 a measure $\gamma\in\widetilde{X}$ and a $\sigma$-finite measure $\rho$ on
 $\widetilde{X}\setminus\{0\}$ satisfying, for all bounded $B\subset X$
 \[
   \int_{\widetilde{X}} \left(\xi (B)\wedge 1\right)\, d\rho(\xi) < \infty,
 \]
 such that $m$ is the distribution of the following random measure:
 \[
     \gamma +  \int_{\widetilde{X}} \xi(\cdot)\, d\omega(\xi),
 \]
 where $\omega$ is a random element of $\left(\widetilde{X}\right)^*$ 
 chosen according to the Poisson measure $\rho^*$ on $\widetilde{X}$ with intensity $\rho$.
 
 The measures $\gamma$ and $\rho$ are uniquely determined by $m$, and $\rho$ is called the 
Lévy measure of $m$.
 \end{theo}

Here we summarize useful properties about the case we are interested in.
\begin{prop}
\label{prop:idsquareintegrable}
Let $(X,\mathcal{A},\mu,T)$ be a conservative ergodic infinite measure preserving system, and let $m$ be the distribution of an ergodic square integrable ID random measure with intensity $\mu$, and whose realizations are almost
surely discrete. Assume that $m$ corresponds to 
$\left(\gamma,\rho\right)$ as in Theorem~\ref{theo:Kallenberg}, then:
\begin{itemize}
\item $\gamma=0$,
\item $\rho$ is $T_{*}$-invariant and supported by $\widetilde{X}_{d}$, 
\item $\int_{\widetilde{X}}\xi\left(A\right)\rho\left(d\xi\right)=\mu\left(A\right)$
for all $A\in\mathcal{A}_{f}$,
\item $\int_{\widetilde{X}}\xi\left(A\right)^{2}\rho\left(d\xi\right)=\int_{\widetilde{X}}\left(\xi\left(A\right)-\mu\left(A\right)\right)^{2}m\left(d\xi\right)$
for all $A\in\mathcal{A}_{f}$,
\item $\left(\widetilde{X},\widetilde{\mathcal{A}},\rho,T_{*}\right)$ has
no $T_{*}$-invariant sets of finite, non zero measure.
\end{itemize}
\end{prop}

\begin{proof}
If $m$ corresponds to $\left(\gamma,\rho\right)$, then $(T_*)_*(m)$ 
is ID too and corresponds to $\bigl(T_*(\gamma),\left(T_{*}\right)_*(\rho)\bigr)$.
Therefore if $m$ is $T_{*}$-invariant then $\gamma$ is $T$-invariant
and $\rho$ is $T_{*}$-invariant. 
For all $A\in\mathcal{A}_{f}$,
\begin{align*}
\mu\left(A\right) & =\int_{\widetilde{X}} \xi(A)\, m(d\xi) \\
 & =\gamma\left(A\right)+
 \int_{\widetilde{X}^*}
 \int_{\widetilde{X}} \xi(A) \, \omega(d\xi)\, \rho^*(d\omega),\\
 &= \gamma\left(A\right)+\int_{\widetilde{X}}\xi\left(A\right)\rho\left(d\xi\right).
\end{align*}
In particular, $\gamma\ll\mu$, and by ergodicity of $\mu$, $\gamma$ is some multiple
of $\mu$. But then $m$ could not be the distribution of a point process
as it would possess a continuous part, unless $\gamma$ vanishes.
From the Poisson construction, we also get that $\rho$ is supported on $\widetilde{X}_{d}$.

Next, the second formula is an application of the isometry formula (\ref{eq:isometry}).


The last fact should come as no surprise for anyone familiar with
ergodic properties of ID systems, however there is no available proof
for this particular case. We give one here:

Assume $\left(\widetilde{X},\widetilde{\mathcal{A}},m,T_{*}\right)$
is ergodic and $\left(\widetilde{X},\widetilde{\mathcal{A}},\rho,T_{*}\right)$
possesses a $T_{*}$-invariant set $K$ with $\rho\left(K\right)<\infty$.
Let $\left(\left(\widetilde{X}\right)^{*},\left(\widetilde{\mathcal{A}}\right)^{*},\rho^{*},\left(T_{*}\right)_{*}\right)$
be the Poisson suspension over $\left(\widetilde{X},\widetilde{\mathcal{A}},\rho,T_{*}\right)$.
The map
\[
\omega\mapsto\int_{\widetilde{X}}\xi\left(\cdot\right) \,\omega\left(d\xi\right)
\]
is a factor map between the suspension and the ergodic system $\left(\widetilde{X},\widetilde{\mathcal{A}},m,T_{*}\right)$.
Moreover we have another factor in the suspension, generated by the stationary process 
\[ \omega\mapsto \left\{\left(T_{*}\right)_{*}^{k}(\omega)\left(K\right)\right\} _{k\in\mathbb{Z}}. 
\]
But as $K$ is $T_{*}$-invariant, $\left(T_{*}\right)_{*}^{k}(\omega)\left(K\right)=\omega(K)$
for all $k\in\mathbb{Z}$, therefore $\left(T_{*}\right)_{*}$ acts
as the identity on this factor. By disjointness between the identity
map and any ergodic map~\cite{Fur67Disj}, we obtain that these two factors are independent
inside the Poisson suspension. In particular, for any $A\in\mathcal{A}_f$,
$\omega\mapsto \int_{\widetilde{X}}\xi\left(A\right) \,\omega\left(d\xi\right)$
and $\omega\mapsto \omega(K)$ are independent. It follows that
\[
\int_{\left(\widetilde{X}\right)^{*}} 
	\left( \int_{\widetilde{X}}\xi\left(A\right) \omega(d\xi) -\mu(A) \right)
	 \Bigl(  \omega(K)-\rho(K)	\Bigr)
	 \rho^*(d\omega)
	     = 0.
%
\]
By the isometry formula (\ref{eq:isometry}), we get:
\[
\int_{\widetilde{X}}\xi\left(A\right)\ind{K}\left(\xi\right)\,\rho\left(d\xi\right)=0
\]
By monotone convergence, we can replace $A$ by $X$ in the preceding integral
and, as $\xi\left(X\right)>0$ $\rho$-a.e., we get that $\ind{K}=0$
$\rho$-a.e., \textit{i.e.} $\rho(K)=0$.
\end{proof}

%

Now, let us explain how $N$ inherits the ID character of $N_\epsilon$ in the limit.
We have $N_\epsilon\ll N$, thus we can introduce the random density $g_\epsilon\egdef \frac{dN_\epsilon}{dN}$. Almost surely, $g_\epsilon$ increases to $1$ as $\epsilon\to 0$.
Let $f$ be a bounded continuous function on $X$ with bounded support. Then by the dominated convergence theorem,
\[
  \int_X f\, dN_\epsilon = \int_X f g_\epsilon \, dN \tend{\epsilon}{0} \int_X f\, dN \quad\text{a.s.}
\]
Let also $h$ be a bounded continuous function on $\RR$. Again by the dominated convergence theorem, we get
\[
  \EE\left[  h\left( \int_X f\, dN_\epsilon \right)  \right] \tend{\epsilon}{0}\EE\left[  h\left( \int_X f\, dN\right)  \right].
\]
By~\cite[Theorem~4.11]{Kallenberg2017}, this characterizes the weak convergence of the distribution of $N_\epsilon$ to the 
distribution of $N$. Using~\cite[Lemma~4.24]{Kallenberg2017}, we can conclude that $N$ is ID as a limit of ID random measures. 

$N$ is thus square integrable and ID, hence admits a Lévy measure $\rho$ which is $\sigma$-finite measure on $\widetilde{X}_d$, and is $T_*$-invariant.  
By ergodicity of $N$,  the system $\left(\widetilde{X},\widetilde{\mathcal{A}},\rho,T_*\right)$ 
enjoys all the properties stated in Proposition~\ref{prop:idsquareintegrable}.

When $T$ has the (P) property, this system has a very simple structure that we fully describe below.

\subsubsection{Infinite quasifactors}

\label{sec:infinite_quasifactors}

We start by extending the notion of quasifactor of Glasner and Meyerovitch to the case of an infinite measure on $\widetilde{X}$.

\begin{definition}
  An \emph{$\infty$-quasifactor} of $(X,\A,\mu,T)$ is a dynamical system $\left(\widetilde{X},\widetilde{\mathcal{A}},\rho,T_{*}\right)$
  where $\rho$ is an infinite, $\sigma$-finite, $T_*$-invariant measure, and 
  \[
    \forall A\in\A_f,\ \int_{\widetilde{X}} \xi(A)\,\rho(d\xi) = \mu(A).
  \]
  The $\infty$-quasifactor $\left(\widetilde{X},\widetilde{\mathcal{A}},\rho,T_{*}\right)$ is said to be \emph{square integrable} if 
  \[
    \forall A\in\A_f,\ \int_{\widetilde{X}} \left(\xi(A)\right)^2\,\rho(d\xi) < \infty.
  \]
\end{definition}

As a simple example of $\infty$-quasifactor, we can consider the infinite measure preserving system $(X^*,\rho,T_*)$,
where $\rho$ is the pushforward of $\mu$ by $x\mapsto\delta_x$. It is actually the Lévy measure of the Poisson point process of intensity $\mu$.
More generally, from Proposition~\ref{prop:idsquareintegrable}, any ergodic infinitely divisible and square integrable $T$-random measure with intensity $\mu$
gives rise to a square integrable $\infty$-quasifactor through its Lévy measure.

\medskip

Let us first consider the case of a \emph{simple $\infty$-quasifactor}, \textit{i.e.} a measure $\rho$ concentrated on the 
set $X^*$ of simple counting measures on $X$.

\begin{prop}
\label{prop:simple}
If $T$ has the (P) property and $\left(X^{*},\mathcal{A}^{*},\rho,T_{*}\right)$
is a simple square integrable $\infty$-quasifactor without $T_{*}$-invariant
set of non-zero finite $\rho$-measure, then $\rho$-a.e. $\xi\in X^{*}$  is concentrated
on a finite subset of a single $T$-orbit. 
\end{prop}

\begin{proof}
As $\left(X^{*},\mathcal{A}^{*},\rho,T_{*}\right)$ is square integrable,
the formula
\[
m\left(A\times B\right):=\int_{X^{*}}\xi\left(A\right)\xi\left(B\right)\rho\left(d\xi\right)
\]
 defines a boundedly finite $T\times T$-invariant measure on $X\times X$. Thanks to porperty~(P), it can be written as
\[
m=\alpha_{\infty}\mu\otimes\mu+\sum_{k\in\mathbb{Z}}\alpha_{k}\Delta_{k},
\]
where $\Delta_k$ is the measure supported on the graph of $T^k$ defined by $\Delta_k(A\times B):=\mu(A\cap T^{-k}B)$.
Observe that, as there is no $T_{*}$-invariant set of non-zero finite
$\rho$-measure,
\[
\frac{1}{n}\sum_{\ell=1}^{n}m\left(A\times T^{-\ell}B\right)
=\frac{1}{n}\sum_{\ell=1}^{n}\int_{X^{*}}\xi\left(A\right)T_{*}^{\ell}\xi\left(B\right)\rho\left(d\xi\right)\to0.
\]
However,
\[
\frac{1}{n}\sum_{\ell=1}^{n}\alpha_{\infty}\mu\otimes\mu\left(A\times T^{-\ell}B\right)=\alpha_{\infty}\mu\otimes\mu\left(A\times B\right)
\]
therefore $\alpha_{\infty}=0$.

This means that $m$ is concentrated on the graphs of the maps $T^k$, $k\in\ZZ$, therefore for $\rho$-a.e. $\xi\in X^{*}$ the product $\xi\otimes\xi$ is concentrated on
these graphs. It follows that $\rho$-a.e. $\xi\in X^{*}$ is concentrated on 
a single $T$-orbit. It remains to verify that $\rho$ is almost surely concentrated on a finite number of points in this orbit.
For this, we observe that Proposition~2.1 in~\cite{sushis} can be generalized to the case of an infinite measure $\rho$ on $X^*$. 
Indeed, the Palm measures of $\rho$ are probability
measures since $\int_{X^{*}}\xi\left(A\right)\rho\left(d\xi\right)=\mu\left(A\right)<+\infty$. 
Since $\rho$ has moments of order~2, the proposition yields that $\rho$-a.e. $\xi\in X^{*}$  is concentrated
on a finite subset of a single $T$-orbit.\end{proof}

\begin{prop}
  \label{prop:cmu}
If $T$ has the (P) property and if $\left(\widetilde{X}_d,\widetilde{\mathcal{A}},\rho,T_{*}\right)$
is a square integrable $\infty$-quasifactor without $T_{*}$-invariant set
of non-zero finite $\rho$-measure, then there exists $c>0$ and a
factor map $\varphi$ from $\left(\widetilde{X}_d,\widetilde{\mathcal{A}},\rho,T_{*}\right)$
to $\left(X,\mathcal{A},c\mu,T\right)$ and some $T_{*}$-invariant
maps $\xi\mapsto a_{k}\left(\xi\right)\ge0$, $k\in\ZZ$, such that
\[
\xi=\sum_{k\in\mathbb{Z}}a_{k}\left(\xi\right)\delta_{T^{k}\varphi\left(\xi\right)}.
\]
\end{prop}

\begin{proof}
First observe that $\{0\}$ is a $T_*$-invariant set. It cannot have infinite $\rho$ measure because $\rho$ is 
$\sigma$-finite, hence $\rho(\{0\})=0$ from the hypotheses.

Let $\xi\in \widetilde{X}_d$ and set $\xi_{\mid \varepsilon}$, $\varepsilon>0$ to be $\xi$ where
we have forgotten points with weights less than $\varepsilon$ and set the other
weights to be $1$. Then $\xi\mapsto\xi_{\mid \varepsilon}$ is a factor map
and $\xi_{\mid \varepsilon}$ turns out to induce on $\widetilde{X}_d\setminus\left\{ \xi_{\mid \varepsilon}=\left\{ 0\right\} \right\} $
a simple square integrable $\infty$-quasifactor without $T_{*}$-invariant
set of non-zero finite measure. Therefore by Proposition~\ref{prop:simple}, $\xi_{\mid \varepsilon}$ has a
finite number of points on its support, which all lie on a single $T$-orbit.

It follows that all the points of the support of $\xi$ are $\rho$-a.s. on a single $T$-orbit, and only a finite number of them 
have a weight greater than any fixed positive constant. We can therefore see that the map $\varphi:\xi\mapsto\varphi\left(\xi\right)$
where $\varphi\left(\xi\right)$ is the point of the support with
the highest weight and the lowest place in the orbit is well defined.
It satisfies
\[
\varphi\left(T_{*}\xi\right)=T\varphi\left(\xi\right).
\]
Now with this ``origin'' $\varphi\left(\xi\right)$, we can define
maps $\xi\mapsto a_{k}\left(\xi\right)\ge0$ so that
\[
\xi:=\sum_{k\in\mathbb{Z}}a_{k}\left(\xi\right)\delta_{T^{k}\varphi\left(\xi\right)}.
\]
We have
\[
T_{*}\xi=\sum_{k\in\mathbb{Z}}a_{k}\left(\xi\right)\delta_{T^{k+1}\varphi\left(\xi\right)}=\sum_{k\in\mathbb{Z}}a_{k}\left(\xi\right)\delta_{T^{k}\varphi\left(T_{*}\xi\right)}
\]
in the one hand, and in the other hand
\[
T_{*}\xi:=\sum_{k\in\mathbb{Z}}a_{k}\left(T_{*}\xi\right)\delta_{T^{k}\varphi\left(T_{*}\xi\right)}.
\]
Therefore the maps $a_{k}$ are $T_{*}$-invariant.

We have for all $A$
\begin{align*}
  \mu(A) & =
\int_{\widetilde{X}_d}\xi\left(A\right)\rho\left(d\xi\right) \\
 & =  \sum_{k\in\mathbb{Z}}\int_{\widetilde{X}_d}a_{k}\left(\xi\right)\delta_{T^{k}\varphi\left(\xi\right)}\left(A\right)\rho\left(d\xi\right)\\
 & =  \sum_{k\in\mathbb{Z}}\int_{\widetilde{X}_d}a_{k}\left(T_{*}^{k}\xi\right)\delta_{\varphi\left(T_{*}^{k}\left(\xi\right)\right)}\left(A\right)\rho\left(d\xi\right)
	\quad\text{by $T_*$-invariance of $a_k$}\\
 & =  \sum_{k\in\mathbb{Z}}\int_{\widetilde{X}_d}\delta_{\varphi\left(\xi\right)}\left(A\right)a_{k}\left(\xi\right)\rho\left(d\xi\right)
	\quad\text{by $T_*$-invariance of $\rho$.}\\
\end{align*}
Let us define for each $k\in\ZZ$ the measure $\rho_k$ by $\frac{d\rho_{k}}{d\rho}\egdef a_{k}$. Then we get 
\[
\mu(A)=\sum_{k\in\mathbb{Z}}\varphi_{*}\rho_{k}\left(A\right),
\]
and in particular $\varphi_{*}\rho_{0}\ll\mu$. But, as $a_{0}>0$ $\rho$-a.e.,
$\rho_{0}\sim\rho$ and we also have $\varphi_{*}\rho\ll\mu$. By ergodicity $\varphi_{*}\rho=c\mu$
for some $c>0$.

Therefore $\varphi$ induces a factor map between $\left(\widetilde{X}_d,\widetilde{\mathcal{A}},\rho,T_{*}\right)$
and $\left(X,\mathcal{A},c\mu,T\right)$.
\end{proof}

\begin{corollary}
\label{cor:product}
Assume that $T$ has the (P) property. Let $\left(\widetilde{X}_d,\widetilde{\mathcal{A}},\rho,T_{*}\right)$ be an infinite
measure-preserving square integrable $\infty$-quasifactor without $T_{*}$-invariant
set of non-zero finite $\rho$-measure. Then there exists a probability measure $\kappa$ on $\mathbb{R}_{+}^{\mathbb{Z}}$ such 
that $\left(\widetilde{X}_d,\widetilde{\mathcal{A}},\rho,T_{*}\right)$ is isomorphic to 
$\left(X\times\mathbb{R}_{+}^{\mathbb{Z}},\mathcal{A}\otimes\mathcal{B}^{\otimes\mathbb{Z}},\mu\otimes(c\kappa),T\times\Id\right)$ (where $c$ is given in Proposition~\ref{prop:cmu}).

Moreover, we have
\[
  \frac{1}{c} = \int_{\mathbb{R}_{+}^{\mathbb{Z}}}  \left(\sum_{k\in\ZZ} a_k\right)\, \kappa \left(d\{a_k\}_{k\in\ZZ}\right).
\]
In particular, $\{a_k\}_{k\in\ZZ}\in\ell_1^+(\ZZ)$ $\kappa$-a.s.
\end{corollary}

\begin{proof}
Define $\Phi$ from $\widetilde{X}_d$ to $X\times\mathbb{R}_{+}^{\mathbb{Z}}$
by
\[
\Phi\left(\xi\right):=\left(\varphi\left(\xi\right),\left(a_{k}\left(\xi\right)\right)_{k\in\mathbb{Z}}\right).
\]
Then $T\times\Id$ preserves $m\egdef\Phi_{*}\rho$ and $\Phi$
is an isomorphism between $\left(\widetilde{X}_d,\widetilde{\mathcal{A}},\rho,T_{*}\right)$
and $\left(X\times\mathbb{R}_{+}^{\mathbb{Z}},\mathcal{A}\otimes\mathcal{B}^{\otimes\mathbb{Z}},m,T\times\Id\right)$.
Since the $\sigma$-algebra generated by $\varphi(\xi)$ is $\sigma$-finite by the preceding proposition, we can disintegrate $m$ with respect to the first coordinate: 
we get a family $(\kappa_x)_{x\in X}$ of probability measures on $\mathbb{R}_{+}^{\mathbb{Z}}$, such that
\[ m(A\times B) = \int_A \kappa_x(B)  c\, d\mu(x),  
\]
where $c$ is given in the preceding proposition. By invariance of $m$ under $T\times\Id$, 
we get $\kappa_x=\kappa_{Tx}$ for $\mu$-almost every $x$, and by ergodicity of $T$ we conclude that there exists $\kappa$ such that $\kappa_x=\kappa$ 
 $\mu$-almost everywhere. This yields $m=\mu\otimes(c\kappa)$.
 
 Now, for each $A\in\mathcal{A}_f$, we have
 \begin{align*}
   \mu(A) &= \int_{\widetilde{X}_d} \xi(A)\,\rho(d\xi) \\
   &= c \int_{\mathbb{R}_{+}^{\mathbb{Z}}} 
   \left( \int_X \sum_{k\in\ZZ} a_k \ind{A}(T^k x) \, \mu(dx) \right) \kappa \left(d\{a_k\}_{k\in\ZZ}\right) \\
   & =  c \, \mu(A) \int_{\mathbb{R}_{+}^{\mathbb{Z}}} \left(\sum_{k\in\ZZ} a_k\right) \, \kappa \left(d\{a_k\}_{k\in\ZZ}\right).
 \end{align*}

 \end{proof}

\subsubsection{End of the proof of Theorem~\ref{thm:discrete}}
We come back to the end of the proof of the main theorem of this section. Recall that under the assumptions of this theorem, 
we were left with the following situation: $N$ is a $T$-random measure with moments of all orders, we showed it is infinitely divisible.
Hence the conclusion follows from the next proposition. 

\begin{prop}\label{squareID}
Assume that $T$ has the (P) property. Let $N$ be a square integrable ID $T$-random measure defined on
some ergodic system $\left(\Omega,\mathcal{F},\PP,S\right)$ whose realizations
are almost surely discrete. Then there exists a probability distribution
$\kappa$ on $\left(\ell_1^+(\ZZ),\mathcal{B}^{\otimes\mathbb{Z}}\right)$
and $c>0$ such that $N$ is distributed as
\[
A\mapsto\int_{X\times \ell_1^+(\ZZ) }\sum_{k\in\mathbb{Z}}a_{k}\ind{A}\left(T^{k}x\right)\mathcal{N}\left(dx,d\left\{ a_{k}\right\} _{k\in\mathbb{Z}}\right)
\]
where $\mathcal{N}$ is a Poisson point process on $X\times\ell_1^+(\ZZ)$
with intensity $c\mu\otimes\kappa$.
\end{prop}

\begin{proof}
By Theorem~\ref{theo:Kallenberg} and Proposition~\ref{prop:idsquareintegrable}, the ID square integrable T-random measure $N$ 
can be described by its Lévy measure $\rho$. The latter is nothing else than a square integrable $\infty$-quasifactor, whose structure is completely
given in Corollary~\ref{cor:product}: the measure space $\left(\widetilde{X}_d,\rho\right)$ is isomorphic to $(X\times \ell_1^+(\ZZ),c\mu\otimes\kappa)$.
In this context, the representation of $N$ as an integral with respect to a Poisson random measure takes the more concrete form explicited in the statement 
of the proposition.  
\end{proof}

\begin{remark}
  Note that the assumption that $N$ has moments of all orders has only been used to obtain the ID character of the random measure. 
  Once this is established, square integrability of $N$ is sufficient to conclude. We do not know if square integrability alone implies the conclusion 
  of Theorem~\ref{thm:discrete}.
\end{remark}

\section{Improved disjointness results}

This last, short, section deals with joinings and disjointness.  The notion of joinings in Ergodic Theory is  the dynamical counterpart of couplings in Probability Theory. It is particularly relevant for the classification of dynamical systems as we do below.
For the reader unfamiliar with this notion, we refer to the seminal paper~\cite{Fur67Disj} and the book~\cite{Glas03Ergojoin}, that present modern ergodic theory through joinings and disjointness.

In~\cite{sushis}, we obtained a series of disjointness results for Poisson suspensions
over transformations satisfying the (P) property with the
additional assumption that the base transformation should have a measurable
law of large numbers, which is a very particular property. We were already convinced that this assumption
was not necessary. The results proved in the present paper allow  to get rid of it. 

The following proposition is an example of how the simplification occurs. 

\begin{prop}
Assume $T$ has the (P) property. If an ergodic probability preserving system $\left(\Omega,\mathcal{F},\PP,S\right)$
is not disjoint from $\left(X^{*},\mathcal{A}^{*},\mu^{*},T_{*}\right)$
then it possesses $\left(X^{*},\mathcal{A}^{*},\left(\alpha\mu\right)^{*},T_{*}\right)$
as a factor for some $\alpha>0$.
\end{prop}

\begin{proof}
  Consider a non-trivial joining $\lambda$  of $\left(\Omega,\mathcal{F},\PP,S\right)$ with the Poisson suspension 
  $\left(X^{*},\mathcal{A}^{*},\mu^{*},T_{*}\right)$, and denote by $\Psi:L^2(\mu^*)\to L^2(\PP)$ the associated Markov operator.
  Denote by $N$ the canonical Poisson $T$-point process defined on $\left(X^{*},\mathcal{A}^{*},\mu^{*},T_{*}\right)$.
  By positivity of $\Psi$, the map $A\in\mathcal{A}_f \mapsto \Psi(N(A))$ extends to a $T$-random measure on $\left(\Omega,\mathcal{F},\PP,S\right)$.
  Indeed, for $A\in\mathcal{A}_f$ we have
  \begin{align*}
    \Psi(N(T^{-1}(A))) & = \Psi (N(A)\circ T_*) \\
    & =  \Psi (U_{T_*}N(A)) \\
    & = U_S \Psi (N(A))\\
    & =  \Psi (N(A)) \circ S,
  \end{align*}
  and 
  $\EE_{\PP}\left[ \Psi(N(A)) \right] = \EE_{\mu^*} \left[N(A)\right] = \mu(A)$.

  Moreover, this $T$-random measure has moments of all orders, as for any $A\in\mathcal{A}_f$ and any $n\ge1$, 
  since $\Psi$ can be interpreted as a conditional expectation, we have
  \[ 
  \EE_{\PP}\left[ \left(\Psi(N(A)\right)^n \right] \le \EE_{\PP}\left[ \Psi\left(N(A)^n\right) \right] = \EE_{\mu^*}\left[N(A)^n\right] <\infty.   
  \]  
  From Proposition~\ref{prop:continuous},  there exists $0\le c\le1$ and a $T$-random measure
$M$ of intensity $\mu$ defined on $\left(\Omega,\mathcal{F},\mathbb{P}\right)$,
supported on discrete measures, such that
\[
\Psi\left(N\left(\cdot\right)\right)=c\mu+\left(1-c\right)M
\]
If $c=1$, then for all $A\in\mathcal{A}_{f}$, $\Psi\left(N\left(A\right)-\mu\left(A\right)\right)=0$,
which means that $\Psi$ vanishes on the first chaos. Let $\Psi^{*}:L^{2}\left(\mathbb{P}\right)\to L^{2}\left(\mu^{*}\right)$
be the adjoint Markov operator, we get that $\Psi^{*}\Psi$ is a Markov
operator on $L^{2}\left(\mu^{*}\right)$ that vanishes on the first
chaos. It can be written as an integral of indecomposable
operators
\[
\Psi^{*}\Psi=\int_{W}\Psi_{w}\rho\left(dw\right),
\]
where $\left(W,\mathcal{W},\rho\right)$ is an auxilliary probability
space. Now the proof follows the same lines as Proposition 4.11 in~\cite{sushis}. We get that $\Psi^{*}\Psi$ is the projection on constants and this implies in turn that the initial joining is trivial, hence a contradiction.

Therefore $c<1$. We deduce that $M$ is a factor of $\left(\Omega,\mathcal{F},\mathbb{P},S\right)$, 
and much as in Section~\ref{subsubsection:removing points}, we obtain a further factor which
is a Poisson point process of intensity $\alpha\mu$ for some $\alpha>0$. (In Section~\ref{subsubsection:removing points}, we got 
such a factor by considering $\left(\Phi\left(N_{\epsilon}\right)\right)_{0}$.)
 \end{proof}

In particular,  following the same proof as in Theorem~5.14 in~\cite{sushis}, we obtain:

\begin{theo}
If $T$ has the (P) property then $\left(X^{*},\mathcal{A}^{*},\mu^{*},T_{*}\right)$
is disjoint from any rank one transformation and any Gaussian dynamical
system.
\end{theo}

\bibliography{sushi}

\providecommand{\bysame}{\leavevmode\hbox to3em{\hrulefill}\thinspace}
\providecommand{\MR}{\relax\ifhmode\unskip\space\fi MR }
\providecommand{\MRhref}[2]{%
  \href{http://www.ams.org/mathscinet-getitem?mr=#1}{#2}
}
\providecommand{\href}[2]{#2}
\begin{thebibliography}{10}

\bibitem{AFS1997}
Terrence Adams, Nathaniel Friedman, and Cesar~E. Silva, \emph{Rank-one weak
  mixing for nonsingular transformations}, Israel J. Math. \textbf{102} (1997),
  269--281.

\bibitem{Ball}
Karen Ball, \emph{Poisson thinning by monotone factors}, Electron. Comm.
  Probab. \textbf{10} (2005), 60--69.

\bibitem{DaleyVereJonesI}
D.J. Daley and D.~Vere-Jones, \emph{An introduction to the theory of point
  processes. {V}ol. {I}}, second ed., Probability and its Applications (New
  York), Springer-Verlag, New York, 2003, Elementary theory and methods.

\bibitem{Danilenko2018}
Alexandre~I. Danilenko, \emph{Infinite measure preserving transformations with
  {Radon} {MSJ}}, Israel J. Math. (2018).

\bibitem{Fur67Disj}
H.~Furstenberg, \emph{Disjointness in ergodic theory, minimal sets and
  diophantine approximation}, Math. Systems Theory \textbf{1} (1967), 1--49.

\bibitem{Glas03Ergojoin}
E.~Glasner, \emph{Ergodic theory via joinings}, American Mathematical Society,
  2003.

\bibitem{Glasner1983}
S.~Glasner, \emph{Quasifactors in ergodic theory}, Israel J. Math. \textbf{45}
  (1983), no.~2-3, 198--208.

\bibitem{HolLyoSoo}
Alexander~E. Holroyd, Russell Lyons, and Terry Soo, \emph{Poisson splitting by
  factors}, Ann. Probab. \textbf{39} (2011), no.~5, 1938--1982.

\bibitem{nfc}
É. Janvresse, E.~Roy, and T.~{de la Rue}, \emph{Nearly finite {Chacon}
  transformation}, hal-01586869, 2017.

\bibitem{sushis}
\bysame, \emph{Poisson suspensions and {S}u{S}his}, Ann. Scient. Éc. Norm.
  Sup. \textbf{50} (2017), no.~6, 1301--1334.

\bibitem{KakPar}
S.~Kakutani and W.~Parry, \emph{Infinite measure preserving transformations
  with ``mixing''}, Bull. Amer. Math. Soc. \textbf{69} (1963), 752--756.

\bibitem{Kallenberg2017}
Olav Kallenberg, \emph{Random measures, theory and applications}, Probability
  Theory and Stochastic Modelling, vol.~77, Springer, Cham, 2017.

\bibitem{Meyerovitch2011}
Tom Meyerovitch, \emph{Quasi-factors and relative entropy for
  infinite-measure-preserving transformations}, Israel J. Math. \textbf{185}
  (2011), 43--60.

\bibitem{Meyerovitch2013}
\bysame, \emph{Ergodicity of {P}oisson products and applications}, Ann. Probab.
  \textbf{41} (2013), no.~5, 3181--3200.

\bibitem{Roy2007}
Emmanuel Roy, \emph{Ergodic properties of {P}oissonian {ID} processes}, Ann.
  Probab. \textbf{35} (2007), no.~2, 551--576.

\end{thebibliography}

\end{document}